\documentclass[11pt, oneside]{article}   	
\usepackage{amsmath, amsthm, amssymb,  amscd}
\usepackage{graphicx}				
\usepackage[left = 1in, right = 1in]{geometry}
	
								\numberwithin{equation}{section}
								\newtheorem{theorem}{Theorem}[section]
								\newtheorem{lemma}[theorem]{Lemma}
								\newtheorem{corollary}[theorem]{Corollary}
								\newtheorem{proposition}[theorem]{Proposition}
								\newtheorem{definition}[theorem]{Definition}

\usepackage{amssymb}
\usepackage{amsmath}

\title{An Erd{\" o}s-Turan Inequality For Compact Simply-Connected Semisimple Lie Groups}
\author{Zev Rosengarten}

\begin{document}

\maketitle

\begin{abstract}

The classical Erd{\" o}s-Turan Inequality bounds how far a sequence of points in the circle is from being equidistributed in terms of its exponential moments. We prove an analogous inequality for all compact simply-connected semisimple Lie groups, bounding how far a sequence is from being equidistributed in the conjugacy classes of the group in terms of the moments of irreducible characters.

\end{abstract}

\section{Introduction}

A very important and interesting notion in the study of sequences on measure spaces is that of equidistribution. For the circle group $S^1$ this is defined as follows.
\begin{definition}
A sequence of points $\{a_i\}_{i = 1}^\infty$ in $S^1$ is said to be equidistributed if for any interval $I \subset S^1$, 
\begin{align*}
\lim_{N\rightarrow \infty} \frac{ |\{1\leq i \leq N: a_i \in I\}|}{N} = m(I),
\end{align*}
where $m$ is mass-one Haar measure on $S^1$.
\end{definition}
That is, a sequence is equidistributed if its empirical distribution tends toward the distribution of a random sequence.
The first major result in the study of equidistribution of general sequences was proved by Hermann Weyl. The Weyl criterion says that a sequence of points in $S^1$ is equidistributed if and only if its exponential moments all go to $0$. Weyl used this criterion to prove equidistribution for various important sequences. Let $e(x)$ be the usual exponential function on $S^1$, so if we write $S^1$ as $\mathbb{R}/\mathbb{Z}$, then $e(x) := e^{2\pi i x}$. Then Weyl's criterion is the following.

\begin{theorem}[Weyl Criterion]
A sequence of points $\{a_i\}_{i = 1}^\infty$ in $S^1$ is equidistributed if and only if for every $0 \neq k \in \mathbb{Z}$,
\begin{align*}
\lim_{N \rightarrow \infty} \frac{1}{N} \sum_{i = 1}^N e(ka_i) = 0.
\end{align*}
\end{theorem}

Weyl proved this result by reformulating equidistribution in terms of functions, as follows.

\begin{definition}
A sequence of points $\{a_i\}$ in $S^1$ is said to be equidistributed if for every continuous function $f: S^1 \rightarrow \mathbb{C}$,
\begin{align*}
\lim_{N \rightarrow \infty} \frac{1}{N} \sum_{i = 1}^N f(a_i) = \int_{S^1} fdm.
\end{align*}
\end{definition}

The Weyl Criterion is then proved by noting that the span of the exponential functions $e(kx)$, $k \in \mathbb{Z}$, is dense in the space $C(S^1)$ of continuous functions from $S^1$ to $\mathbb{C}$. Now the classical Erd{\" o}s-Turan Inequality gives a quantitative version of the Weyl Criterion, by bounding how far a sequence is from equidistribution in terms of its exponential moments. This measure of how far a sequence is from equidistribution is given by the discrepancy.

\begin{definition}
Let $a = \{a_i\}_{i = 1}^N$ be a sequence of points in $S^1$. We define the discrepancy of $a$, denoted $D(a)$, by
\begin{align*}
D(a) := \sup_{I \subset S^1}  \left| \frac{1}{N}|\{1 \leq i\leq N: a_i \in I\}| - m(I) \right|,
\end{align*}
where the supremum is over all intervals $I \subset S^1$.
\end{definition}
\noindent The Erd{\" o}s-Turan Inequality then asserts the following.

\begin{theorem}[Erd{\" o}s-Turan Inequality]
Let $a = \{a_i\}_{i = 1}^N$ be a sequence of points in $S^1$. Then for some absolute constant $C$, and every positive integer $k$,
\begin{align*}
D(a) \leq  C \left(\frac{1}{k} + \sum_{0 < |n| \leq k} |\frac{1}{N}\sum_{i = 1}^N e(nx)| \right).
\end{align*}
\end{theorem}

There is also a  generalization of this inequality to any torus $(S^1)^m$ (the Erd{\"o}s-Turan-Koksma Inequality; see [D-T], page 15, Theorem 1.21).

In this paper we prove an analogous inequality for every compact simply-connected semisimple Lie group. In the same way that the Erd{\"o}s-Turan Inequality is a quantitative version of the Weyl Criterion, our result may be thought of as a quantitative version of the Peter-Weyl theorem, since the generalization of the Weyl Criterion to a compact Lie group $G$ follows from the Peter-Weyl Theorem, which tells us that the characters of representations of $G$ span a dense subspace of $C^{\mbox{class}}(G)$, the space of continuous complex-valued class functions on $G$. Our result bounds how far a sequence is from being equidistributed in the set of conjugacy classes in terms of the moments of irreducible characters.  Our techniques generalize those of Niederreiter [N], who essentially proves such an inequality for the group $SU(2)$ ([N], Lemma 3), although he doesn't phrase his results in this way.

The structure of this paper is as follows. In section 2 we introduce and state our main theorem. In section 3 we study the pushforward by fundamental characters of Haar measure on a compact simply-connected semisimple Lie group. In section 4  we prove our main theorem.

\section{The Main Theorem} 
In this section we introduce and state our main theorem.
Let $G$ be a compact simply-connected semisimple Lie group. Since $G$ is simply-connected, its representations are in one-to-one correspondence with the representations of its Lie algebra, via the map that takes a representation of $G$ to its differential at the identity. It then follows from the general representation theory of semisimple Lie algebras that 
  the (complex) representation ring of $G$ is a polynomial ring in irreducible representations $\rho_1$, ..., $\rho_{r_1}$, $\gamma_1$,..., $\gamma_{r_2}$, $\gamma_1^*$,..., $\gamma_{r_2}^*$(If $\beta$ is a representation, then $\beta^*$ denotes the dual representation.), where the $\rho_i$ are self-dual (equivalently, their characters are real-valued), the $\gamma_i$ are not, and $r_1 + 2r_2 = r$, where $r = rank(G)$ is the dimension of a maximal torus in $G$. These representations are called fundamental representations of $G$. We let $\chi_1$, ..., $\chi_{r_1}$, $\Gamma_1$,..., $\Gamma_{r_2}$, $\bar{\Gamma}_1$,..., $\bar{\Gamma}_{r_2}$ denote their characters. We use these characters to map elements of $G$ to $\mathbb{R}^{r_1} \times \mathbb{C}^{r_2} = \mathbb{R}^r$ (via the identification $\mathbb{C} = \mathbb{R}^2$) by letting $\chi_j$ map to $x_j$ and $\Gamma_j$ to $z_j = \alpha_j + i\beta_j$, and we call this map $P$. That is, we define $P: G\rightarrow \mathbb{R}^{r_1} \times \mathbb{C}^{r_2} = \mathbb{R}^r$ by
 \begin{align*}
 P(g) := (\chi_1(g), ..., \chi_{r_1}(g), \Gamma_1(g), ..., \Gamma_{r_2}(g) ).
 \end{align*} 
We now define the degree of a tensor power of fundamental representations. Given the fundamental representations $\delta_1$, ..., $\delta_r$ of $G$, we call $\delta_1^{\otimes m_1} \otimes ... \otimes \delta_r^{\otimes m_r}$ a tensor power of the fundamental representations of degree $\sum_{i = 1}^r {m_i}$. Furthermore, this tensor power decomposes uniquely as a direct sum of irreducible representations, and we say that an irreducible representation appears in this tensor power if it appears as a summand in this decomposition. An equivalent way of formulating this is as follows. Let $\alpha_1$, ..., $\alpha_r$ denote the fundamental weights of $G$ (that is, the highest weights of the fundamental representations of $G$). Then the irreducible representation with highest weight $\lambda = \sum_{i = 1}^r m_i \alpha_i$ appears in a tensor power of the fundamental representations of degree $\leq d$ if and only if $\sum_{i = 1}^r m_i \leq d$. Now as a function on $P(G)$ any irreducible character is a polynomial in the $x_j$, $\alpha_j$, $\beta_j$. Furthermore, any polynomial in these variables of degree $d$ is in the $\mathbb{C}$-span of the (pushforwards of the) irreducible characters appearing in the tensor powers of the fundamental representations of degree at most $d$, since any polynomial in these variables is a polynomial in the $x_j$, $z_j$, $\bar{z_j}$ of the same degree via the relations $\alpha_j = (z_j + \bar{z_j})/2$ and $\beta_j = (z_j - \bar{z_j})/2i$, and a monomial in these variables is just the character of the corresponding tensor power of the fundamental representations. Now the image of $G$ under this map is a compact set, contained in the box $[-M, M]^r\subset \mathbb{R}^r$, where $M := \max_{\delta} dim(\delta)$, the maximum being taken over the fundamental representations $\delta$ of $G$. We denote the pushforward of Haar measure to this box by $\mu$. That is,
 \begin{align*}
 \mu := P_*(dg),
 \end{align*}
 where $dg$ is Haar measure on $G$, normalized to have total mass one. Thus, the measure $\mu$ on $\mathbb{R}^r$ is defined by $\mu(A) = dg(P^{-1}(A))$ for any Borel set $A \subset \mathbb{R}^r$.
 Then we define the discrepancy of a sequence in $G$ by the usual definition of discrepancy for the pushforward sequence with respect to $\mu$. That is, given a sequence $g_1$, ..., $g_N \in G$, if its image in $[-M, M]^r$ under $P$ is $a = \{a_i = P(g_i)\}_{i = 1}^N$, then we define the discrepancy $D(g)$ of $\{g_i\}_{i = 1}^N$ by
\[
D(g) = D(a) := \displaystyle \sup_{I\subset [-M, M]^r} \left|\mu(I) - \frac{1}{N}|\{i: a_i\in I\}| \right|,
\]
where the supremum is over all boxes $I\subset [-M, M]^r$ (A box in $\mathbb{R}^r$ is defined to be a product of intervals.). To understand this discrepancy, it suffices to understand the so-called star-discrepancy $D^*$ defined by 
\[
D^*(g) = D^*(a) : = \displaystyle \sup_{x\in [-M, M]^r} \left|\mu(I_x) - \frac{1}{N}|\{i: a_i\in I_x\}|  \right|,
\]
where for $x = (x_i)\in [-M, M]^r$, $I_x: = \prod_{i = 1}^r {[-M, x_i]}$. Indeed, we have the inequalities
\begin{align}
D^*(a) \leq D(a) \leq 2^rD^*(a). \label{eq:ineq3}
\end{align}
The first inequality is obvious. To see the second, consider a box $\prod_{i = 1}^r {[a_i, b_i]}\subset [-M, M]^r$. For each subset $J\subset \{1,..., r\}$, define $x_J\in [-M, M]^r$ by 
\[
(x_J)_i := 
\begin{cases}
a_i & \mbox{if } i\in J \\
b_i & \mbox{if } i\notin J \\
\end{cases}.
\]
Then for any measure $\lambda$ on $[-M, M]^r$, we have (by inclusion-exclusion)
\[
\lambda(\prod_{i = 1}^r {[a_i, b_i]}) = \sum_{J\subset \{1,..., r\}} {(-1)^{|J|}\lambda(I_{x_J})},
\]
where $I_y$ is as defined previously. Applying this to the measures $\mu$ and $\frac{1}{N}\sum_{i = 1}^N {\delta_{a_i}}$ and subtracting yields the second inequality in (\ref{eq:ineq3}). Our proof will therefore utilize the star-discrepancy, since we lose little information. We now state our main result.
\\

\begin{theorem} 
\label{maintheorem}
Let $G$ be a compact simply-connected semisimple Lie group. Let $g = \{g_i\}_{i = 1}^N$ be a sequence in $G$. Then for every positive integer $k$,
\[
D(g) \leq C_G \left(\frac{1}{k} + \sum_{\chi} {\left|  \frac{1}{N}\sum_{i = 1}^N {\chi(g_i)} \right|}\right),
\]
where
\begin{align*}
C_G = \frac{600r^2(10M)^r 2^{r_2 + \frac{1}{2}(r + dim(G))}}{\pi^r},
\end{align*}
$r$ is the rank of $G$, $r_2$ is the number of non-self-dual pairs of dual fundamental representations, $M$ is the maximum dimension of any fundamental representation of $G$, and
where the sum is over the characters $\chi$  of the nontrivial irreducible complex representations appearing in all tensor powers of the fundamental representations of $G$ of degree at most $2(1 + \lfloor \frac{r}{2} \rfloor) (k - 1) + r$.
\end{theorem}

We remark that the constant $C_G$ can be improved. For example, the $10^r$ can be replaced with $(4 + \epsilon)^r$ for any positive $\epsilon$, at the price of increasing $C_G$ by a factor depending on $\epsilon$, but independent of $G$. We have not in this paper made a substantial effort to optimize $C_G$, since the most important parameter in Theorem \ref{maintheorem} is $k$, because it is allowed to vary for fixed $G$.
\\
In the next section we study the pushforward $\mu$ of Haar measure. The reader wishing to turn immediately to the proof of Theorem \ref{maintheorem} may skip to section 4, temporarily taking Proposition \ref{pushforward} for granted, since this is the only result of the next section that is used the proof of Theorem \ref{maintheorem}.

\section{The Pushforward of Haar Measure}

Let $G$ be a compact simply-connected semisimple Lie group. Then the representations of $G$ are in one-to-one correspondence with the representations of its Lie algebra $\mathfrak{g}$ via the correspondence $\Gamma \mapsto d\Gamma$, where $d\Gamma$ is the differential of $\Gamma$ at the identity. This correspondence gives the following commutative diagram:
\begin{center}
$
\begin{CD}
\mathfrak{g}   @> d\Gamma >>  \mathfrak{gl}(V)\\
@VVexpV    @VVexpV\\
G   @>\Gamma>> GL(V)
\end{CD}
$
\end{center}
In particular, from the representation theory of semisimple Lie algebras, we know that the representation ring of $G$ is a polynomial ring in certain irreducible representations, called fundamental representations, which are the irreducible representations with highest weights equal to the dual basis  of the basis of some fixed Cartan subalgebra $\mathfrak{h}\subset \mathfrak{g}$ consisting of the coroots $H_\alpha$ corresponding to the simple roots $\alpha$. Now we denote the fundamental representations of $G$ by $\delta_1$, ..., $\delta_{r_1}$, $\gamma_1$,..., $\gamma_{r_2}$, $\gamma_1^*$,..., $\gamma_{r_2}^*$  (If $\beta$ is a representation, then $\beta^*$ denotes the dual representation.), where the $\alpha_i$ are self-dual (equivalently, their characters are real-valued), the $\gamma_i$ are not, and $r_1 + 2r_2 = r$, where $r = rank(G)$ is the rank of $G$. We let $\chi_1$, ..., $\chi_{r_1}$, $\Gamma_1$,..., $\Gamma_{r_2}$, $\bar{\Gamma}_1$,..., $\bar{\Gamma}_{r_2}$ denote their characters. We use these characters to define a map from $G$ to $\mathbb{R}^{r_1} \times \mathbb{C}^{r_2} = \mathbb{R}^r$ (via the identification $\mathbb{C} = \mathbb{R}^2$) by letting $\chi_j$ map to $x_j$ and $\Gamma_j$ to $z_j = \alpha_j + i\beta_j$, and we call this map $P$. That is, we define $P: G\rightarrow \mathbb{R}^{r_1} \times \mathbb{C}^{r_2} = \mathbb{R}^r$ by
\begin{align}
\label{eq:pushmap}
P(g) := (\chi_1(g), ..., \chi_{r_1}(g), \Gamma_1(g), ... , \Gamma_{r_2}(g)).
\end{align}
We note that $P$ is a class function that is injective on the set of conjugacy classes. That $P$ is a class function is clear. That $P$ is injective on the set of conjugacy classes is a consequence of the Peter-Weyl Theorem. Indeed, if $g$, $h\in G$ and $P(g) = P(h)$, then since every character is a polynomial in the fundamental characters, $\chi(g) = \chi(h)$ for every character $\chi$ of $G$. Thus, by Peter-Weyl, $f(g) = f(h)$ for every continuous class function $f: G \rightarrow \mathbb{C}$, so $g$ and $h$ are conjugate.
Now our main concern in this section is to study the pushforward $\mu$ of Haar measure $dg$ on $G$ (normalized to have mass one) by $f$. That is,
\begin{align*}
\mu = P_*(dg).
\end{align*}
Thus the measure $\mu$ is defined by $\mu(B) := dg(P^{-1}(B))$ for every Borel set $B \subset \mathbb{R}^r$.
Our goal in this section 
is to prove that this pushforward measure has continuous density with respect to Lebesgue measure. That is, we wish to prove that 
\begin{align*}
\mu = FdX,
\end{align*}
where $dX$ is Lebesgue measure on $\mathbb{R}^r$ and $F: \mathbb{R}^r \rightarrow \mathbb{R}$ is a continuous function. We will actually do more, by giving a formula for $F$. 
First, we make a few observations and fix some notation. Let $A = P(G) \subset \mathbb{R}^r$ denote the image of $G$ under $f$. Now since $P$ is a class function, if we choose a maximal torus $T\subset G$  (fixed from here on out), we lose no information by considering the restriction $P: T\rightarrow \mathbb{R}^r$, since every element of $G$ is conjugate to an element of $T$. It is mainly this restricted map with which we shall work. Now two elements of $T$ are conjugate in $G$ if and only if they lie in the same orbit under the action of the Weyl group $W$. Thus, $P$ descends to a continuous map $P: T/W \rightarrow A$. Our first simple but crucial observation is the following.

\begin{lemma}
\label{homeomorphism}
$P:T/W \rightarrow P(G)$ is a homeomorphism.
\end{lemma}

\begin{proof}
$P: T/W \rightarrow A := P(G)$ is continuous, and it is surjective by definition. It's injective, because $P$ is injective on the set of conjugacy classes in $G$, and two elements of $T$ are conjugate in $G$ if and only if they lie in the same orbit under the Weyl group action. Thus, $P: T/W \rightarrow A$ is a continuous bijection. Since any continuous bijection from a compact space to a Hausdorff space is a homeomorphism, this proves the lemma.
\end{proof}

We return now to looking at $P$ as a map from $T$ to $A \subset \mathbb{R}^r$. Since our goal is to understand the pushforward of Haar measure, we must first understand the Haar measure that $T = (\mathbb{R}/2\pi \mathbb{Z})^r$ inherits from $G$. For this we have the Weyl integration formula. Let $d\Theta = d\theta_1...d\theta_r$, where $d\theta$ is the usual Lebesgue measure on $\mathbb{R}/2\pi \mathbb{Z}$, and let $R^+$ denote the positive roots of $\mathfrak{g}$, the Lie algebra of $G$. Then we have the following.

\begin{lemma}[Weyl Integration Formula]
\label{weylintegrationformula}
Let $f: G \rightarrow \mathbb{R}$ be a continuous class function. Then
\begin{align*}
\int_G fdg = (\frac{1}{2\pi})^r \int_{T/W} \left| \prod_{\alpha \in R^+} (e^{\alpha/2} - e^{-\alpha/2})\right| ^2fd\Theta.
\end{align*}
\end{lemma}

For a proof, see [B], p. 334, Corollary 1.
Now in pushing this measure forward by $P$, we must change variables from $d\Theta$ to $dX$. This involves a Jacobian. So let $Jac := (\frac{\partial X}{\partial \Theta})$ be the Jacobian matrix of the map $P$, and let $J(\Theta)$ denote $det(Jac)$. We then have $d\Theta = \frac{dX}{|J(\Theta)|}$, so the pushforward measure $\mu$ on $A$ is given, using the Weyl Integration Formula, by

\begin{align}
\label{eq:measure1}
\mu = \frac{1}{(2\pi)^r}\frac{\left| \displaystyle \prod_{\alpha \in R^+} (e^{\alpha/2} - e^{-\alpha/2})\right|^2}{|J|}dX.
\end{align}

Since $P: T/W \rightarrow A$ is a homeomorphism, a continuous function on $T$ gives a continuous function on $A$. There are therefore two obstacles to $F$ being continuous (where $F$, as before, is the density function of $\mu$, i.e. the expression appearing before the $dX$ in (\ref{eq:measure1})). The first potential problem is that the Jacobian may vanish somewhere without the numerator vanishing to a high enough order to cancel it out. The second is that, even if $F$ is continuous on $A$, in order for it to extend to a continuous function on $\mathbb{R}^r$, we need it to vanish on $\partial A$, the boundary of $A$, since it vanishes outside of $A$. We shall show that in fact neither of these problems is an issue. This follows easily from the following formula for the Jacobian, plus our determination of $f^{-1}(\partial A)$.

\begin{proposition}
\label{jacobian}
$|J| = 2^{-r_2} \left| \displaystyle \prod_{\alpha \in R^+} (e^{\alpha/2} - e^{-\alpha/2}) \right|$.
\end{proposition}

As an immediate consequence of this, we shall give an exact formula for the pushforward of Haar measure at the end of this section. We wait until then to record the formula because by then we will also be able to prove continuity of the density function of $\mu$ with respect to Lebesgue measure, which is not immediately obvious from the formula for $\mu$.
The proof of Proposition \ref{jacobian} will occupy us for the rest of this section. 
First, we define an order relation on the set $L$ of integral weights of $G$ as follows. Let $\mathfrak{h} = Lie(T)$ be the Cartan subalgebra corresponding to our choice of maximal torus $T$. Recall that in defining positive roots, one chooses a hyperplane $H \subset \mathfrak{h}^*$, and declares an element of $\mathfrak{h}^*$ to be positive if it lies on one side of $H$ and negative if it lies on the other side. If we choose $H$ (as can certainly be done) so that $H \cap L = \{0\}$, then this induces a total ordering on $L$. For $\alpha$, $\beta \in L$, we write $\alpha > \beta$ if $\alpha$ is greater than $\beta$ with respect to this order relation.
Now we give an important definition.

\begin{definition}
An exponential polynomial $Q$ is a sum of the form $Q = \sum_{\lambda \in L} c_\lambda e^\lambda$, where the sum is over the integral weights $\lambda$ of $\mathfrak{g}$, $c_\lambda \in \mathbb{C}$, and $c_\lambda = 0$ for all but finitely many $\lambda$. The degree of $Q$, denoted $deg(Q)$, is the  weight $\lambda$ such that $c_\lambda \neq 0$ and $c_\beta = 0$ for all $\beta > \lambda$.
\end{definition}
We can consider exponential polynomials as functions on the torus or as formal expressions. There is no essential difference between these two points of view because the homomorphism from the ring of exponential polynomials to the ring of functions on $T$ given by sending a polynomial to its corresponding function is injective. Indeed, the independence of the exponential monomials with different weights follows from their orthogonality.
Now recall that any element $w\in W$ has a length $l(w)$, well-defined modulo two, and defined by saying that $w$ is the product of $l(w)$ reflections in simple roots. We define $(-1)^w$ to be $(-1)^{l(w)}$.
We then define an action of the Weyl group on exponential polynomials by setting, for $w\in W$ and $Q = \sum c_\lambda e^\lambda$,
\begin{align*}
w(Q) = \sum c_\lambda e^{w\lambda}.
\end{align*}
We call an exponential polynomial invariant (respectively, anti-invariant) provided that $w(Q) = Q$ (respectively, $w(Q) = (-1)^wQ$) for every $w\in W$. Now let $\Pi$ denote the set of simple roots of $\mathfrak{g}$. We then define as usual the integral weight $\rho$ to be
\begin{align*}
\rho := \frac{1}{2}\sum_{\alpha \in R^+} \alpha = \sum_{\beta \in \Pi} \beta.
\end{align*}
Our first step toward determining $J$ is the following lemma.

\begin{lemma}
\label{anti-invariance of J}
$J$ is an anti-invariant exponential polynomial of degree $\leq \rho$.
\end{lemma}

\begin{proof}
We first check that $J$ is an exponential polynomial of degree $\leq \rho$. For each $\beta \in \Pi$, let $\chi_\beta$ denote the fundamental character with highest weight equal to the weight $\lambda_\beta$ defined by 
\begin{align*}
\lambda_\beta(H_\alpha) = 
\begin{cases}
1 & \mbox{if}\quad \alpha = \beta \\
0 & \mbox{if}\quad \beta \neq \alpha \in \Pi, \\
\end{cases}
\end{align*}
where $H_\alpha$ is the coroot of the simple root $\alpha$. 
We have 
\begin{align*}
J = 
det\left[(\frac{\partial \chi_\beta}{\partial \theta_i})_{\chi_\beta \; \mbox{real}} \quad \left(\frac{\partial}{\partial \theta_i} \frac{1}{2}(\chi_\beta + \bar{\chi_\beta})\right)_{\chi_\beta \; \mbox{complex}}  \quad  \left(\frac{\partial}{\partial \theta_i} \frac{1}{2i}(\chi_\beta - \bar{\chi_\beta}\right)_{\chi_\beta \; \mbox{complex}}  \right]_{1 \leq i \leq r}.
\end{align*}
Now taking a factor of $1/2$ out of each column of the second matrix and a factor of $i/2$ out of each column of the third,  then subtracting the $\beta$ column of the third matrix from the $\beta$ column of the second, pulling a factor of two out of the $\beta$ column of the second, and finally adding the $\beta$ column of the second matrix to the $\beta$ column of the third, we get that 
\begin{align}
\label{eq:jacobfactor}
J = (\frac{i}{2})^{r_2}  det( \frac{\partial \chi_\beta}{\partial \theta_i})_{\beta \in \Pi, 1 \leq i \leq r}.
\end{align}
Now $\chi_\beta = e^{\lambda_\beta} +$ terms of lower degree. $\lambda_\beta$ is a linear form in the $\theta_i$, so $
\frac{\partial \chi_\beta}{\partial \theta_i} = c_\beta e^{\lambda_\beta} +$ terms of lower degree for some constant $c_\beta$. From the usual formula for the determinant of a square matrix as a polynomial in the entries, it follows that $J$ is indeed an exponential polynomial of degree $\leq \sum_{\beta \in \Pi} \beta = \rho$. To see that $J$ is anti-invariant, we recall that $f$ is invariant under the Weyl group, hence for any reflection $w$ in a root, we have that $f\circ w = f$, hence $f$ and $f\circ w$ have the same Jacobian. But the Jacobian of $f\circ w$ is, by the chain rule, $-w(J)$, since the Jacobian of a reflection is $-1$. Thus, $w(J) = -J$ for any reflection in a root, from which it follows that $J$ is anti-invariant, since such reflections generate the Weyl group.
\end{proof}

\begin{lemma}
\label{manifold}
$T/W$ is a (topological) $r$-manifold with boundary. The boundary $\partial(T/W)$ of $T/W$ is 
\begin{align*}
\partial(T/W) = \{\Theta: e^{\alpha(\Theta)} = 1 \; \mbox{for some } \alpha \in R^+\}.
\end{align*}
\end{lemma}

\begin{proof}
Let $\mathfrak{h} = Lie(T)$. Then we have the exponential map $exp: \mathfrak{h} \rightarrow T$. Let $x \in \mathfrak{h}$, $\Theta = exp(x)$. Suppose first that $\Theta = exp(x) \neq exp(b)$ for any $b$ in the boundary of a Weyl chamber. The set of points $y$ such that $exp(y) = exp(b)$ for some $b$ in the boundary of a Weyl chamber is a union of discrete translates of finitely many hyperplanes. It therefore follows that $x$ has a neighborhood $U$ such that $V = exp(U)$ satisfies $w(V) \cap V = \phi$ for all $1 \neq w \in W$. Then since $exp$ is a local diffeomorphism, shrinking $U$ if necessary, $V$ is a neighborhood of $\Theta$ with this same property. It then follows that $\Theta \in int(T/W)$ in the sense of interior of manifolds, that is, it has a neighborhood homeomorphic to an $r$-disk. Suppose, on the other hand, the $x$ lies in the boundary of a Weyl chamber. Then there is a small disk $U$ about $x$ upon which a nontrivial subgroup $\tilde{W}$ of the Weyl group acts, and $U/\tilde{W}$ is homeomorphic to the quotient of a disk by a group generated by reflections, hence is homeomorphic to half-disk with $x$ in its boundary. It follows, since $exp$ is a local homeomorphism, that a sufficiently small neighborhood $V$ of $\Theta = exp(x)$ satisfies the same properties, that is, $\Theta$ has no neighborhood in $T/W$ homeomorphic to a disk, but it has one homeomorphic to a half-disk. Thus, $\Theta \in \partial(T/W)$, where boundary is taken in the sense of manifolds. Finally, the image under $exp$ of the boundaries of Weyl chambers is precisely the set $\{\Theta: e^{\alpha(\Theta)} = 1 \; \mbox{for some } \alpha \in R^+\}$, so the proof of the lemma is complete.
\end{proof}

\begin{lemma}
\label{boundary}
$P: T/W \rightarrow P(G)$  satisfies $P^{-1}(\partial(P(G))) = \partial (T/W)$.
\end{lemma}

\begin{proof}
Since $P: T/W \rightarrow A := P(G)$ is a homeomorphism by Lemma \ref{homeomorphism}, and since $T/W$ is an $r$-manifold with boundary, it suffices to show that being in the interior of a manifold (versus in the boundary) is a topological property. This can be seen in many ways. Here's one: Let $M$ be an $r$-manifold. If $x\in int(M)$, then for any neighborhood $U$ of $x$, $H_{r-1}(U-\{x\}) \neq 0$, while if $x \in \partial (M)$, then $x$ has a neighborhood with $H_{r-1}(U-\{x\}) = 0$. Thus, being in the interior, or boundary, is a purely topological property, hence is preserved by $P$.
\end{proof}

\begin{corollary}
\label{J vanishes on boundary}
$J$ vanishes on $P^{-1}(\partial(P(G))) = \cup_{\alpha \in R^+} \{\Theta: e^{\alpha(\Theta)} = 1\}$.
\end{corollary}

\begin{proof}
$T$ is a manifold without boundary, hence by the inverse function theorem, if $J(\Theta) \neq 0$, then $J(\Theta) \subset int(A)$. Thus, $J$ vanishes on $P^{-1}(\partial A)$. That $P^{-1}(\partial A)$ is the set claimed in the lemma follows from Lemmas \ref{boundary} and \ref{manifold}.
\end{proof}

\begin{lemma}
\label{divisible}
$J = Q*\prod_{\alpha \in R^+} (e^{\alpha/2} - e^{-\alpha/2})$ for some exponential polynomial $Q$.
\end{lemma}

\begin{proof} For each simple root $\beta \in \Pi$, let $X_\beta := e^{\lambda_\beta}$. Then the ring of exponential polynomials is just the ring of Laurent polynomials $\mathbb{C}[X_\beta, X_\beta^{-1}: \beta \in \Pi]$. By the previous corollary, $J$ is an element of this ring such that setting $X_{\beta} = 1$ makes $J = 0$. Thus, $J$ is divisible by $X_\beta - 1$ for any simple root $\beta$. 
But since any root can be mapped to a simple root by an element of the Weyl group, anti-invariance of $J$ implies that $J$ is divisible by $e^{\alpha} - 1$ for each root $\alpha$. Now I claim that if $\alpha \neq \pm \beta$ are roots, then $e^{\alpha} - 1$ and $e^{\beta} - 1$ are pairwise coprime, from which it will follow that $J$ is divisible by $\prod_{\alpha \in R^+} (e^{\alpha} - 1) =   
e^{\rho} \prod_{\alpha \in R^+} (e^{\alpha/2} - e^{-\alpha/2})$, hence by $\prod_{\alpha \in R^+} (e^{\alpha/2} - e^{-\alpha/2})$, which will complete the proof. To see the claim, let $\alpha \neq \pm \beta$. The action of an element of the Weyl group induces an automorphism of the ring of exponential polynomials, and since any root may be mapped by the Weyl group to a simple root, we may assume that $\alpha$ is simple. But in that case, saying that $P(X_\gamma : \gamma \in \Pi)$ is relatively prime to $X_\alpha - 1$ says simply that setting $X_\alpha = 1$ doesn't make $P$ equal to $0$. But this is clear for $P = e^{\beta} - 1$, since $\beta$ is not a multiple of $\alpha$.
\end{proof}

We can now finish the proof of Proposition \ref{jacobian}. Let $Q$ be as in Lemma \ref{divisible}. Then since $J$ and $\prod_{\alpha \in R^+} (e^{\alpha/2} - e^{-\alpha/2})$ are both anti-invariant exponential polynomials, it follows that $Q$ is an invariant exponential polynomial. Since $deg(J) \leq deg(\prod_{\alpha \in R^+} (e^{\alpha/2} - e^{-\alpha/2})) = \rho$, it follows that $deg(Q) \leq 0$. Since $Q$ is invariant and any nonzero weight maps by an element of the Weyl group to a positive (even a dominant) weight, it follows that $Q$ has no terms $e^{\lambda}$ with $\lambda < 0$. Thus, $Q = c_0e^{0} = c_0$ for some constant $c_0$, that is, $Q$ is constant. So 
\begin{align*}
J = C \prod_{\alpha \in R^+} (e^{\alpha/2} - e^{-\alpha/2})
\end{align*}
for some constant $C$, and it remains to compute $|C|$.  We do this by comparing the leading coefficients of both sides of the above equation. By (\ref{eq:jacobfactor}), we get that the absolute value of the leading coefficient of $J$ is
\begin{align*}
|C| = 2^{-r_2} |det(\frac{\partial \lambda_\beta}{\partial \theta_i})_{\beta \in \Pi, 1 \leq i \leq r}|.
\end{align*}
I claim that the above determinant has absolute value one, from which it will follow that $|C| = 2^{-r_2}$, which will complete the proof of Proposition \ref{jacobian}. We check that the determinant is one in absolute value as follows. We have a linear map $\ell : T = (\mathbb{R}/2\pi\mathbb{Z})^r \rightarrow (\mathbb{R}i/2\pi i \mathbb{Z})^r$ given by $\ell(\Theta) := (\lambda_\beta(\Theta))_{\beta \in \Pi}$, where $\Theta := (\theta_1$, ..., $\theta_r)$. I claim that $\ell$ is an isomorphism, from which it follows that its determinant has absolute value one, and the proof will be complete. First, the linear forms $\lambda_\beta$, $\beta \in \Pi$, are linearly independent, hence $\ell$ is surjective. Now we check that $\ell$ is injective. Suppose that $\ell(\Theta) \in (2\pi i \mathbb{Z})^r$. Then in particular, $e^{\beta(\Theta)} = 1$ for every weight $\beta$ of $G$. But then $\chi(\Theta) = \chi(0)$ for every character $\chi$ of $G$ (Here $0$ denotes the identity of $T$ = the identity of $G$.). Therefore, by Peter-Weyl, $\Theta$ is conjugate in $G$ to $0$, hence $\Theta = 0$. Thus, $\ell$ is injective, hence is an isomorphism, which completes the proof.
\hfill $\blacksquare$
\\

\noindent We now immediately obtain the main result of this section.

\begin{proposition}
\label{pushforward}
Let $r$ be the rank of $G$, $r_2$ the number of complex conjugate pairs of nonreal fundamental characters (equivalently, the number of dual pairs of non-self-dual representations). The pushforward of normalized Haar measure by the map $P: G\rightarrow \mathbb{R}^r$ given by equation (\ref{eq:pushmap}) is given by $\mu = FdX$, where $dX = dx_1...dx_r$ is Lebesgue measure on $\mathbb{R}^r$, and
\begin{align*}
F(x) = 
\begin{cases}
\frac{2^{r_2}}{(2\pi)^r} \left| \prod_{\alpha \in R^+} (e^{\alpha/2} - e^{-\alpha/2}) \right|, & x\in P(G)\\
0, & x\notin P(G).\\
\end{cases}
\end{align*}
Furthermore, $F$ is a continuous function on $\mathbb{R}^r$.
\end{proposition}

\begin{proof}
That $F$ is given by the claimed formula is immediate from (\ref{eq:measure1}) and Proposition \ref{jacobian}. All that remains is to check that $F$ is continuous.
As usual, let $A = P(G)$. Since the expression for $F(x)$ for $x\in A$ is a continuous 
function of $P^{-1}(x)$, and since $P: T/W \rightarrow A$ is a homeomorphism (Lemma 
\ref{homeomorphism}), it follows that $F|_A$ is continuous. $F$ is clearly continuous on $ext(A)$, 
so it remains to show that $F$ is continuous on $\partial A$, or equivalently, that $F|_{\partial A} = 
0$. Now $\partial A = \cup_{\alpha \in R^+} \{e^\alpha = 1\}$ by Lemmas \ref{boundary} and \ref{manifold}, and the formula for $F|_A$ vanishes at these points. Thus, $F|_{\partial A} = 0$, so $F$ is continuous.
\end{proof}

Although we will not use it, we record the following interesting consequence of the results of this section.

\begin{corollary}
\label{homdiff}
$P: T/W \rightarrow P(G)$ is a homeomorphism and $P: int(T/W) \rightarrow int(P(G))$ is a diffeomorphism.
\end{corollary}

\begin{proof}
The first assertion is Lemma \ref{homeomorphism}. The second is a consequence of the Inverse Function Theorem and our formula for the Jacobian of $P$, which shows that it vanishes only at $\partial (T/W) = \cup_{\alpha \in R^+} \{e^{\alpha} = 1\}$.
\end{proof}

\section{Proof of The Main Theorem}

We note that it suffices to prove the assertion for $k$ odd (with a smaller constant, say $\frac{1}{2}C_G$ instead of $C_G$), and that we may work with $D^*$ instead of $D$ by using inequality (\ref{eq:ineq3}). We work with the pushforward sequence $a = \{a_i = P(g_i)\}_{i = 1}^N$ in $[-M, M]^r$. Our proof of Theorem \ref{maintheorem} involves two major linchpins. One is representing the moments of $a$ by certain integral operators, and the other involves constructing a suitable kernel. It is to the first task that we now turn. We require a couple of preparatory lemmas.

\begin{lemma}
\label{measuredifferentiable}
Let $H: [-M, M]^m \rightarrow \mathbb{R}$ be continuous, and consider the measure $\lambda = HdX$, where $dX$ is Lebesgue measure. For $x \in [-M, M]^m$, let $I_x := \prod_{i = 1}^m [-M, x_i]$, and for $J \subset \{1, ..., m\}$, define, $x_J \in \prod_{i \in J} [-M, M]$ to be the projection of $x$ onto the $J$ coordinates, that is, $(x_J)_i = x_i$ for $i \in J$.
Then for every set $J\subset \{1, ..., m\}$, 
\begin{align*}
\frac{\partial \lambda(I_x)}{\partial X_J}(x) = \int_{\prod_{i \notin J} [-M, x_i]} H(x_J, y)dy_{J^c},
\end{align*}
where $\frac{\partial}{\partial X_J} = \prod_{i \in J} \frac{\partial}{\partial x_i}$ and $dy_{J^c} = \prod_{i \notin J} dy_i$ is Lebesgue measure on $\prod_{i \notin J} [-M, M]$. In particular, $\frac{\partial \lambda(I_x)}{\partial X_J}$ exists and is continuous.
\end{lemma}

\begin{proof}
We proceed by induction on $|J|$, the case $|J| = 0$ being trivial. Now suppose that $J \neq \phi$, say $j \in J$, and that the lemma is true for $J - \{j\}$. We will prove the lemma for $J$. We have by the induction hypothesis that 
\begin{align}
\frac{\partial \lambda(I_x)}{\partial X_{J-\{j\}}} = \int_{\prod_{i\notin J - \{j\}} [-M, x_i]} H(x_{J - \{j\}}, y)dy_{(J-\{j\})^c}.
\label{eq:4.1}
\end{align}
By Fubini's Theorem, the integral on the right side of (\ref{eq:4.1}) can be written as 
\begin{align*}
\int_{\prod_{i\notin J} [-M, x_i]} \int_{-M}^{x_j}H(x_{J- \{j\}}, y_j, y_{J^c})dy_jdy_{J^c}.
\end{align*}
Thus, taking $\frac{\partial}{\partial x_j}$ of both sides of (\ref{eq:4.1}) yields
\begin{align*}
\frac{\partial \lambda(I_x)}{\partial X_J} = \lim_{\epsilon \rightarrow 0} \int_{\prod_{i\notin J} [-M, x_i]} \frac{1}{\epsilon} 
\int_{x_j}^{x_j + \epsilon} H(x_{J- \{j\}}, y_j, y_{J^c})dy_jdy_{J^c} = \int_{\prod_{i\notin J} [-M, x_i]} H(x_J, y)dy_{J^c},
\end{align*}
where the last equality comes from the uniform continuity of $H$. Thus, the lemma is true for $J$, 
which completes the induction.
\end{proof}

Now we need some notation. For $x = (x_i)\in [-M, M]^m$, let $I_x := \prod_{i = 1}^m {[-M, x_i]}$. Let $J\subset \{1,..., m\}$.
For a smooth function $h:[-M, M]^m\rightarrow \mathbb{C}$, define a function $h_J: \prod_{k\in J} {[-M, M]} \rightarrow \mathbb{C}$ as follows: For $x = (x_k)\in \prod_{k\in J} {[-M, M]}$, define $x_J \in [-M, M]^m$ by
\[
(x_J)_k = 
\begin{cases}
x_k & \mbox{if } k\in J \\
M & \mbox{if } k\notin J \\
\end{cases}
.
\]
Then define $h_J(x) := h(x_J)$. Now suppose that we have a measure $\lambda = HdX$ on $[-M, M]^m$, where $H$ is a continuous function and $dX$ is Lebesgue measure. Define a measure $\lambda_J$ on $\prod_{k\in J} {[-M, M]}$ by $
\lambda_J = \frac{\partial}{\partial X_J} (\lambda(I_{x_J})) dX_J$, where $\frac{\partial}{\partial X_J}
$ is shorthand for $\prod_{k\in J} {\frac{\partial}{\partial x_k}}$ and $dX_J$ denotes Lebesgue 
measure on $\prod_{k\in J} {[-M, M]}$. This makes sense by Lemma \ref{measuredifferentiable}.
Finally, define a function $R^{\lambda}_J: \prod_{k \in J} {[-M, M]} \rightarrow \mathbb{R}$ by
\[
R^{\lambda}_J(x) := \frac{1}{N}|i: a_{iJ} \in I_{x_J}| - \lambda_J((I_J)_{x}).
\]
In the above expression, $(I_J)_x$ denotes $\prod_{i \in J} [-M, x_i]$. Note that $\lambda_J((I_J)_x) = \lambda(I_{x_J})$, and that
the above expression for $R^{\lambda}_J$ is indeed equal to $R^{\lambda}_J$ in the previously defined sense if we take $R^{\lambda} := R^{\lambda}_{\{1,..., m\}}$.

\begin{lemma} 
\label{integralopsgeneral}
Let $H: [-M, M]^m \rightarrow \mathbb{R}$ be a continuous function, $\lambda = HdX$, where $dX$ is Lebesgue measure on $[-M, M]^m$. For any smooth function $h: [-M, M]^m \rightarrow \mathbb{C}$, 
\[
\sum_{\phi \neq J\subset \{1,..., m\}} (-1)^{|J|}\int_{\prod_{k\in J} {[-M, M]}} {R^{\lambda}_J(x)\frac{\partial h_J}{\partial x_J}dX_J} = 
 \frac{1}{N}\sum_{i = 1}^N {h(a_i)} - \int_{[-M, M]^m} {hd\lambda} .
\]
\end{lemma}

\begin{proof} We may assume that $h$ is real-valued, by considering its real and imaginary parts. Note that the $J = \phi$ term in the sum is $0$, so we may add it in without modifying anything. We proceed by induction on $m$. In the $m = 0$ case the assertion is trivial. Now suppose the assertion is true for values less than $m$ and we'll prove it for $m$. For convenience denote $\{1,..., m\}$ by $J_m$. Also, for $y = (y_k) \in [-M, M]^m$ let $B_y$ denote the box $\prod_{k = 1}^m {[y_k, M]}$. We have
\begin{align}
\int_{[-M, M]^m} {R^{\lambda}_{J_m}\frac{\partial h}{\partial X_{J_m}}dX_{J_m}} = 
\frac{1}{N}\sum_{i = 1}^N {\int_{B_{a_i}} {\frac{\partial}{\partial x_1}...\frac{\partial}{\partial x_m} h dx_1...dx_m}} - \int_{[-M, M]^m} {\lambda(I_x) \frac{\partial}{\partial x_1}...\frac{\partial}{\partial x_m} h dx_1...dx_m}.  \label{eq:int1}
\end{align}
We use Fubini's Theorem to write each integral in the first sum as an iterated integral over the $x_i$ and then use the Fundamental Theorem of Calculus repeatedly to evaluate it. We obtain
\[
\int_{B_{a_i}} {\frac{\partial}{\partial x_1}...\frac{\partial}{\partial x_m} h dx_1...dx_n} = 
\sum_{J\subset \{1,..., m\}} {(-1)^{|J|}h_J((a_i)_J)}.
\]
We evaluate the last integral on the right side of (\ref{eq:int1}) by integrating by parts once in each variable in order to move all of the derivatives from $h$ to $\lambda(I_x)$. This is permissible by Lemma \ref{measuredifferentiable}. At each step we get a boundary term and a term in which the derivative has been moved. The boundary term on the left (that is, the boundary term at -M) is $0$, since it's the $\lambda_J$ measure of a box of the form $\prod_{k \in A} [d_k, e_k] \times \{-M\}$ for some set $A$, which is $0$. Thus, if we let $J$ be the set of variables for which we move the derivative (i.e., for which we take the term coming from integration by parts in $x_j$ that is not the boundary term), then we see that the last integral in (\ref{eq:int1}) equals
\[
\sum_{J\subset \{1,..., m\}} (-1)^{|J|}\int_{\prod_{k\in J} [-M, M]} {h_J\frac{\partial \lambda(I_x)}{\partial X_J}}.
\]
Thus, putting everything together, we obtain
\begin{align}
\int_{[-M, M]^m} {R^{\lambda}_{J_m}\frac{\partial h}{\partial X_{J_m}}dX_{J_m}} =
\sum_{J\subset \{1,..., m\}} {(-1)^{|J|} \left[ \frac{1}{N} \sum_{i = 1}^N {h_J((a_i)_J)} - \int_{\prod_{k \in J} [-M, M]} {h_J\frac{\partial \lambda(I_x)}{\partial X_J}} \right] }.
\label{eq:int2}
\end{align}
Now we apply the induction hypothesis to the terms of the sum on the right side of (\ref{eq:int2}). To do this we first note that if $I\subset J\subset \{1,..., m\}$, then $(h_J)_I = h_I$ and $(\lambda_J)_I = \lambda_I$. Using this and the induction hypothesis we get that for $J\subsetneq \{1,..., m\}$, the $J$ summand on the right side of (\ref{eq:int2}) equals
\[
(-1)^{|J|}\sum_{I\subset J} (-1)^{|I|}{\int_{\prod_{k \in I} [-M, M]} {R^{\lambda}_I\frac{\partial h_I}{\partial x_I}dX_I}}.
\]
Thus, for $I\neq \{1,..., n\}$, the coefficient of $\int_{\prod_{k \in I} [-M, M]} {R^{\lambda}_I\frac{\partial h_I}{\partial x_I}dX_I}$ on the right side of (\ref{eq:int2}) equals
\[
(-1)^{|I|}\sum_{I\subset J\subsetneq\{1,..., m\}} {(-1)^{|J|}}.
\]
Grouping the terms in the above sum by $j = |J|$, and letting $i := |I|$, we see that the sum equals
\[
(-1)^i\sum_{j = i}^{m-1} {(-1)^j{m-i \choose j-i}} = \sum_{j = 0}^{m-1-i} {(-1)^j{m-i \choose j}} = (1-1)^{m-i} - (-1)^{m-i} = -(-1)^{m-i}.
\]
Therefore, rearranging (\ref{eq:int2}) gives (using the fact that, by Lemma \ref{measuredifferentiable}, $\frac{\partial^m \lambda(I_x)}{\partial x_1...\partial x_m} = H$)
\[
(-1)^m \left[ \frac{1}{N}\sum_{i = 1}^N {h(a_i)} - \int_{[-M, M]^n} {hd\lambda} \right] =
 \sum_{J\subset \{1,..., m\}} {(-1)^{m - |J|}\int_{\prod_{k \in I} [-M, M]} {R^{\lambda}_I\frac{\partial h_I}{\partial x_I}dX_I}},
\]
and multiplying by $(-1)^m$, we see that the lemma is true for $m$, which completes the induction. 
\end{proof}

Taking m = r and $\lambda = \mu$, the pushforward measure from sections 2 and 3, in Lemma \ref{integralopsgeneral}, and using Proposition \ref{pushforward}, we obtain the following.

\begin{corollary}
\label{corollary}
Let $\mu = P_*(dg)$ be the pushforward of Haar measure in $[-M, M]^r$ (See section 3.). Then For any smooth function $h: [-M, M]^r \rightarrow \mathbb{C}$, 
\[
(-1)^r\int_{[-M, M]^r} {R^{\mu}(x)  \frac{\partial^rh}{\partial x_1...\partial x_r}dX} = 
 \frac{1}{N}\sum_{i = 1}^N {h(a_i)} - \int_{[-M, M]^r} {hd\mu} .
\]
\end{corollary}

Now we turn to our second main task, the construction of a suitable kernel. For this we make use of the Chebyshev polynomials of the first kind (suitably normalized for our purposes) defined by
\[
T_k(x) := cos(k\,arccos(\frac{x}{2M\sqrt{r}}))\; for\; x\in [-2M\sqrt{r}, 2M\sqrt{r}].
\]
Then $T_k$ is a polynomial of degree $k$. Now define as usual for $x\in \mathbb{R}^r$, $|x| := 
\sqrt{\sum_{i =1}^r {x_i^2}}$. Then we will make use of the following polynomials on $\mathbb{R}^r
$ (whose crucial property is that they're heavily concentrated at the origin.):
\[
f_k(x) := (\frac{T_k(|x|)}{|x|})^{2(1 + \lfloor \frac{r}{2} \rfloor)} \; for\; |x| \leq 2M\sqrt{r}.
\]
Since $k$ is odd, $T_k(0) = 0$, and $f_k$ is an even polynomials in $|x|$, hence is a polynomial in $|x|^2$ of degree $(1 + \lfloor \frac{r}{2} \rfloor) (k-1)$, hence is indeed a polynomial in the $x_i$ of degree $2(1 + \lfloor \frac{r}{2} \rfloor)(k-1)$. It is manifestly centrally symmetric and nonegative. The reason for the $2M\sqrt{r}$ normalization in our definition of $T_k$ is to ensure that for any $v \in [-M, M]^r$, $f_k(x-v)$ is defined on all of $[-M, M]^r$. The following lemma is a quantitative version of the statement that $f_k$ is concentrated at the origin.

\begin{lemma}
\label{kernel}
$f_k$ satisfies the following properties.

\noindent (i) For any $0 < c \leq \frac{6}{5}M\sqrt{r}$, 
\[
(\frac{1}{5})^{2(1+ \lfloor \frac{r}{2} \rfloor)}  \frac{c^rk^{2(1 + \lfloor \frac{r}{2}\rfloor) - r}vol(S^{r-1})}{r(M\sqrt{r})^{2(1+\lfloor \frac{r}{2}\rfloor)}} \leq
\int_{|x| \leq ck^{-1}} {f_kdx}     \leq \frac{c^rk^{2(1 + \lfloor \frac{r}{2}\rfloor) - r}vol(S^{r-1})}{r(M\sqrt{r})^{2(1+\lfloor \frac{r}{2}\rfloor)}},
\]
where $S^{r-1}$ is the $(r-1)$-sphere.

\noindent (ii) For any $0 < t \leq 2M\sqrt{r}$, $\int_{t \leq |x| \leq 2M\sqrt{r}} {f_kdx} \leq 2 vol(S^{r-1})t^{r - 2(1 + \lfloor \frac{r}{2} \rfloor)}$.
\end{lemma}

\begin{proof} The first assertion we obtain by considering the Taylor expansion of $T_k$ about $0$. Suppose that $|y| \leq \frac{6}{5}\frac{M\sqrt{r}}{k}$. We have, by Taylor's Theorem,
\begin{align}
\nonumber arccos(\frac{y}{2M\sqrt{r}}) = \pi/2 - \frac{y}{2M\sqrt{r}} + E(y) \quad \mbox{as} \quad y\rightarrow 0,
\end{align}
where 
\begin{align}
|E(y)| \leq \frac{1}{2} \frac{(\frac{|y|}{2M\sqrt{r}})^3}{{(1 - (\frac{y}{2M\sqrt{r}})^2)}^{3/2}}
\leq \frac{3}{16}\frac{y}{M\sqrt{r}},
\label{eq:error}
\end{align}
where in the last inequality we used the fact that $\frac{|y|}{2M\sqrt{r}} \leq \frac{3}{5}$.
Thus,
\begin{align}
\nonumber T_k(y) = cos(k\;arccos(\frac{y}{2M\sqrt{r}})) = cos(k\pi/2 - \frac{ky}{2M\sqrt{r}} + kE(y))
= \pm sin(\epsilon),
\end{align}
where $\epsilon := \frac{ky}{2M\sqrt{r}} - kE(y)$.
By (\ref{eq:error}),
\begin{align*}
\frac{5}{16}\frac{ky}{M\sqrt{r}}  \leq  \epsilon  \leq \frac{11}{16}\frac{ky}{M\sqrt{r}}.
\end{align*}
In particular, $\epsilon \leq \frac{33}{40}$, and since $\frac{16}{25}\theta \leq sin(\theta) \leq \theta$ for $0 \leq \theta \leq \frac{33}{40}$, it follows that
\begin{align}
\label{eq:estimate2}
\frac{1}{5}\frac{k}{M\sqrt{r}}  \leq   |\frac{T_k(y)}{y}|   \leq \frac{k}{M\sqrt{r}}  \quad \mbox{for} \quad 
|y| \leq \frac{6}{5}\frac{M\sqrt{r}}{k}.
\end{align}
Now changing to polar coordinates, we have 
\begin{align}
\int_{|x| \leq ck^{-1}} {f_kdx} = vol(S^{r-1}) \int_{0}^{ck^{-1}} {(T_k(\rho)/\rho)^{2(1 + \lfloor \frac{r}{2} \rfloor)}\rho^{r - 1}d\rho}.
\label{eq:estimate3}
\end{align}
By (\ref{eq:estimate2}), if $c \leq \frac{6}{5}M\sqrt{r}$, then
\begin{align}
(\frac{k}{5M\sqrt{r}})^{2(1 + \lfloor \frac{r}{2} \rfloor)} \int_{0}^{ck^{-1}} {\rho^{r-1}d\rho} \leq  
\int_{0}^{ck^{-1}} {(T_k(\rho)/\rho)^{2(1 + \lfloor \frac{r}{2} \rfloor)}\rho^{r - 1}d\rho}
\leq   (\frac{k}{M\sqrt{r}})^{2(1 + \lfloor \frac{r}{2} \rfloor)} \int_{0}^{ck^{-1}} {\rho^{r-1}d\rho}.
\label{eq:estimate4}
\end{align}
Since
\[
\int_{0}^{ck^{-1}} \rho^{r-1}dr = \frac{c^rk^{-r}}{r},
\]
(i) follows from (\ref{eq:estimate3}) and (\ref{eq:estimate4}). As for (ii), we note that $|T_k(\rho)| \leq 1$ for all $\rho$, hence
\begin{align*}
\int_{|x| > t} {f_kdx} = vol(S^{r-1}) \int_{t}^{2M\sqrt{r}} {(T_k(\rho)/\rho)^{2(1 + \lfloor \frac{r}{2} \rfloor)}\rho^{r - 1}d\rho} \leq
vol(S^{r-1})\int_{t}^{2M\sqrt{r}} {\rho^{r-1-2(1 + \lfloor \frac{r}{2} \rfloor)}d\rho} 
\end{align*}
\begin{align*}
\leq 2 vol(S^{r-1})t^{r-2(1 + \lfloor \frac{r}{2} \rfloor)},
\end{align*}
where we have used the fact that $(M\sqrt{r})^{r - 2(1 + \lfloor \frac{r}{2} \rfloor)} \leq t^{ r - 2(1 + \lfloor \frac{r}{2} \rfloor)}$ since $r - 2(1 + \lfloor \frac{r}{2} \rfloor) < 0$.
\end{proof}

Now we must construct an antiderivative of $f_k$ satisfying certain useful properties. For this we make use of the following simple lemma.

\begin{lemma} 
\label{antiderivative}
Let $f\in \mathbb{R}[x_1,..., x_r]$ be a polynomial of degree $m$. Then there exists a polynomial $h$ of degree $m + r$ with the following properties:

\noindent (i) $\frac{\partial}{\partial x_1}...\frac{\partial}{\partial x_r}h = f$.

\noindent (ii) $h_J(x) = 0$ for all $J\subsetneq \{1,..., r\}$.
\end{lemma}

\begin{proof}
Any polynomial satisfying (i) is automatically of degree $\geq m + r$, so we only need to ensure that $deg(h) \leq m + r$. There clearly exists a polynomial $\tilde{h}$ of degree $m + r$ satisfying (i). Then I claim that 
$h := (-1)^r\sum_{J \subset \{1,..., r\}} {(-1)^{|J|}\tilde{h}_J}$, which is clearly of degree at most $deg(\tilde{h}) \leq 
m + r$, satisfies (i) and (ii). Here, although technically $\tilde{h}_J \in \mathbb{R}[x_i: i \in J]$, we consider 
the $\tilde{h}_J$ as elements of $\mathbb{R}[x_1,..., x_r]$. First, since $\frac{\partial}{\partial x_1}...
\frac{\partial}{\partial x_r} \tilde{h}_J = 0$ for $J \neq \{1,..., r\}$, $h$ satisfies (i). Now for $I$, $J \subset 
\{1,..., r\}$, $(\tilde{h}_J)_I = \tilde{h}_{I\cap J}$. Therefore, 
\begin{align}
h_I = (-1)^r\sum_{J \subset \{1,..., r\}} (-1)^{|J|}h_{I\cap J}.
\label{eq:eq33}
\end{align}
Now if $K \subset I \neq \{1,..., r\}$, then grouping by $|J|$, we see that the coefficient of $h_K$ on the right side of (\ref{eq:eq33}) is
\[
(-1)^r \sum_{i = 0}^{r - |I|} {(-1)^{i + |K|}{r - |I| \choose i}} = (-1)^{r + |K|} (1-1)^{r - |I|} = 0.
\]
Hence $h_I = 0$, so $h$ satisfies (ii).
\end{proof}

Now we may construct the desired antiderivative. Let $v \in [-M, M]^r$. Consider the translation 
$f_v(x) := f_k(x - v)$ of our polynomial $f_k$ from before. It is a polynomial of degree $2(1 + \lfloor \frac{r}{2} \rfloor) 
(k-1)$. Let $h$ be a polynomial as in Lemma \ref{antiderivative} applied to $f = f_v$. Then $h$ has degree $2(1+ \lfloor \frac{r}{2} \rfloor) (k-1) + r$. Therefore we may write
\begin{align}
h = \sum_{\chi} {C_{\chi} \chi},
\end{align}
where $C_{\chi}\in \mathbb{C}$ and the sum is over the (pushforwards of the) irreducible characters $\chi$ appearing in all tensor powers of fundamental representations of degree at most 
$2(1+\lfloor \frac{r}{2}\rfloor)(k-1) + r$. We have the following.

\begin{lemma}
\label{coefficientbound}
$|C_{\chi}| \leq  3 vol(S^{r-1})(M\sqrt{r})^{r - 2(1 + \lfloor \frac{r}{2} \rfloor)}k^{2(1 + \lfloor \frac{r}{2} \rfloor) - r}$.
\end{lemma}

\begin{proof}
First, by Lemma \ref{kernel}, (i) and (ii) (with $t = M\sqrt{r}k^{-1}$), we have
\begin{equation}
\int_{|x-v|\leq 2M\sqrt{r}} {f_k(x-v)} \leq 3 vol(S^{r-1})(M\sqrt{r})^{r - 2(1 + \lfloor \frac{r}{2} \rfloor)}k^{2(1 + \lfloor \frac{r}{2} \rfloor) - r}.  
\label{eq:no7}
\end{equation}
Now for any $x = (x_i)\in [-M, M]^r$, let $B_x := \prod_{i = 1}^r [x_i, M]$. Then by repeatedly applying Fubini's Theorem we obtain
\begin{align*}
(-1)^r\sum_{J\subset \{1,..., r\}} {(-1)^{|J|}h_J(x)} = \int_{B_x} {\frac{\partial}{\partial x_1}...\frac{\partial}{\partial x_r} h dX} = \int_{B_x} {f_k(x-v)dX} 
\end{align*}
\begin{align}
\leq 3  vol(S^{r-1})(M\sqrt{r})^{r - 2(1 + \lfloor \frac{r}{2} \rfloor)}k^{2(1 + \lfloor \frac{r}{2} \rfloor) - r}
\label{eq:no8}
\end{align}
by (\ref{eq:no7}). By property (ii) in Lemma \ref{antiderivative}, every term in the sum on the left side 
of (\ref{eq:no8}) is $0$ except for the $J = \{1,..., r\}$ term. Thus, we obtain
\begin{align}
||h||_{\infty} \leq 3  vol(S^{r-1})(M\sqrt{r})^{r - 2(1 + \lfloor \frac{r}{2} \rfloor)}k^{2(1 + \lfloor \frac{r}{2} \rfloor) - r}.   
\label{eq:no9}
\end{align}
Now we may finish the argument. We have by the orthogonality of characters
\begin{align*}
|C_{\chi}| = |\int {h\bar{\chi}d\mu}| & \leq ||h||_{\infty} \int {|\chi|d\mu} \\
& \leq ||h||_{\infty} (\int {|\chi|^2d\mu})^{1/2} \quad \mbox{(by the Cauchy-Schwarz Inequality)} \\
& = ||h||_{\infty} \quad \mbox{(by the orthonormality of characters)} \\
& \leq 3  vol(S^{r-1})(M\sqrt{r})^{r - 2(1 + \lfloor \frac{r}{2} \rfloor)}k^{2(1 + \lfloor \frac{r}{2} \rfloor) - r}     \quad \mbox{(by (\ref{eq:no9}))}. 
\end{align*}
\end{proof}

We now have all of the tools necessary to complete the proof of Theorem \ref{maintheorem}. First, some notation. Let $D^* := D^*(g) = D^*(P(g))$. Let $R(x) := R^{\mu}_{\{1,..., r\}}(x)$ (Recall our definition of the $R_J$ from before.). As before, given $x = (x_i)\in [-M, M]^r$, let $I_x := \prod_{i = 1}^r [-M, x_i]$. For any measure $\lambda$ on $[-M, M]^r$, define a function $H_\lambda: [-M, M]^r \rightarrow \mathbb{R}$ by $H_\lambda(x) := \lambda(I_x)$. Now we note the following inequality, an immediate consequence of Proposition \ref{pushforward}:
\begin{align}
\label{eq:supnormF}
||F||_\infty \leq \frac{2^{r_2 + \frac{1}{2}(dim(G) - 3r)}}{\pi^r}.
\end{align}
Since $H_\lambda$ is Lipschitz for $\lambda = dX$ with Lipschitz constant $2^{2r-1}M^{r-1}$, and since $\mu = FdX$, it follows from (\ref{eq:supnormF}) that $H := H_\mu$ is Lipschitz with Lipschitz constant 
\begin{align}
\label{eq:lipschitz}
\frac{2^{r_2  - 1 + \frac{1}{2}(dim(G) + r)}M^{r-1}}{\pi^r}.
\end{align}
 Given $x = (x_i)$, $y = (y_i)\in \mathbb{R}^r$, we write $x < y$ if $x_i < y_i$ for all $i$. Now either $R(y) <  -\frac{3}{4}D^*$ for some $y \in [-M, M]^r$, or $R(y) > \frac{3}{4}D^*$ for some $y $. Suppose the former is true. Then I claim that there exists $y \in [-M, M]^r$ such that
\begin{align}
R(y)  <  -\frac{3}{4}D^* \quad \mbox{and} \quad y_i - (-M) \geq \frac{3}{4}\frac{\pi^r  2^{\frac{1}{2}(r - dim(G)) - r_2 + 1}}{M^{r-1}}    D^* \quad \mbox{for all} \;  1\leq i \leq r.    \label{eq:now1}
\end{align}
By assumption there exists $y$ satisfying the first inequality, and I claim that $y$ then automatically satisfies the second inequality. Indeed, if $y_i - (-M) < hD^*$, then $R(y) \geq           -H(I_y) \geq -(2M)^{r-1}hD^*||F||_\infty$, and so choosing $h = \frac{3}{4} \frac{1}{(2M)^{r-1}||F||_\infty}$, we then obtain a contradiction to $R(y) < -\frac{3}{4}D^*$. The claim then follows from (\ref{eq:supnormF}).
Now given $x \in [-M, M]^r$, $x < y$, we have
\begin{align*}
-D^* \leq R(x) \leq -\frac{3}{4} D^* - R(y) + R(x) \leq -\frac{3}{4} D^* + H(y) - H(x)
\end{align*}
\begin{align}
 \leq -\frac{3}{4} D^* + \frac{2^{r_2 - 1 + \frac{1}{2}(dim(G) + r)}M^{r-1}}{\pi^r} |y - x|,
\label{eq:now2}
\end{align}
since $H$ is Lipschitz with Lipschitz constant given by (\ref{eq:lipschitz}).
Together,
(\ref{eq:now1}) and (\ref{eq:now2}) imply that for  
\begin{align}
c := \frac{1}{4} \frac{\pi^r}{M^{r-1}(1 + \sqrt{r})2^{r_2 - 1 + \frac{1}{2}(dim(G) + r)}}, 
\label{eq:constant7}
\end{align}
$v := y - cD^*(1,..., 1)$
 is such that
\begin{align}
x \in [-M, M]^r \quad \mbox{and} \quad -D^* \leq R(x) \leq -\frac{1}{2} D^* \quad \mbox{whenever} \quad |x - v| \leq cD^*.   
\label{eq:12}
\end{align}
In the same way, we obtain a similar such $v$ (with $D^*$ instead of $-D^*$ in (\ref{eq:12})) in the case that $R(y) > \frac{3}{4}D^*$ for some $y$.
Now if $D^* \leq Bk^{-1}$, where 
\begin{align}
B = \frac{40Mr^{3/2}5^r}{c}, 
\end{align}
then Theorem \ref{maintheorem} is trivially true. So let us assume from now on that $D^* >  Bk^{-1}$. Now we consider the functions $f_v$ and $h$ from before. We have for the $c$ in (\ref{eq:constant7}),
\begin{align}
\nonumber & |\int_{[-M, M]^r} {R(x)f_v(x)dX}| = |\int_{|x - v| \leq cD^*} {R(x)f_k(x-v)dX} + \int_{|x-v|>cD^*, x\in [-M, M]^r} {R(x)f_k(x-v)dX}| \\
& \geq 
\frac{1}{2}D^*|\int_{|x - v| \leq cD^*} {f_k(x-v)dX}| - D^*|\int_{|x-v|>cD^*, \; x\in [-M, M]^r} {f_k(x-v)dX}|
\label{eq:13}
\end{align}
by (\ref{eq:12}).
Now since $cD^* > cBk^{-1} > M\sqrt{r}k^{-1}$, Lemma \ref{kernel}, (i) implies that the first integral on the right satisfies
\begin{align}
\label{eq:number21}
\int_{|x - v| \leq cD^*} {f_k(x-v)dX}
\geq (\frac{1}{5})^{2(1 + \lfloor \frac{r}{2} \rfloor)}\frac{k^{2(1 + \lfloor \frac{r}{2} \rfloor) - r} vol(S^{r-1})}{r(M\sqrt{r})^{2(1 + \lfloor \frac{r}{2}\rfloor) - r}}.
\end{align}
On the other hand, by Lemma \ref{kernel}, (ii), the second integral on the right side of (\ref{eq:13}) satisfies
\begin{align}
\label{eq:number27}
\int_{|x-v|>cD^*, \; x\in [-M, M]^r} {f_k(x-v)dX} \leq 2vol(S^{r-1})(cD^*)^{r - 2(1 + \lfloor \frac{r}{2} \rfloor)}
\leq 2vol(S^{r-1})  (cBk^{-1})^{r - 2(1 + \lfloor \frac{r}{2} \rfloor)},
\end{align}
where we've used the fact that $r - 2(1 + \lfloor \frac{r}{2} \rfloor) < 0$.
It follows from the definition of $B$, (\ref{eq:number21}), and (\ref{eq:number27}) that
\begin{align*}
\int_{|x - v| \leq cD^*} {f_k(x-v)dX}
\geq 4  \int_{|x-v|>cD^*, \; x\in [-M, M]^r} {f_k(x-v)dX}.
\end{align*}
Thus, by (\ref{eq:13}),
\begin{align*}
 |\int_{[-M, M]^r} {R(x)f_v(x)dX}| \geq \frac{1}{4}D^* \int_{|x - v| \leq cD^*} {f_k(x-v)dX}.
 \end{align*}
It now follows from (\ref{eq:number21}) that
\begin{align}
\label{eq:final11}
 |\int_{[-M, M]^r} {R(x)f_v(x)dX}| \geq 
 \frac{1}{4} (\frac{1}{5})^{2(1 + \lfloor \frac{r}{2} \rfloor)}\frac{k^{2(1 + \lfloor \frac{r}{2} \rfloor) - r} vol(S^{r-1})}{r(M\sqrt{r})^{2(1 + \lfloor \frac{r}{2}\rfloor) - r}}   D^*
 \end{align}
On the other hand, we can bound the integral on the left in terms of the moments of irreducible representations. Indeed, we have
\begin{align}
(-1)^r\int_{[-M, M]^r} {R(x)f_v(x)dX} = \sum_{\phi \neq J\subset \{1,..., r\}} (-1)^{|J|}\int_{\prod_{k\in J} {[-M, M]}} {R_J(x)\frac{\partial h_J}{\partial x_J}dX_J},   
\label{eq:15}
\end{align}
since $h$ satisfies (i) and (ii) of Lemma \ref{antiderivative} with $f = f_v$. Now the right side of (\ref{eq:15}) equals
\[
\sum_{\chi}C_{\chi} {\sum_{\phi \neq J\subset \{1,..., r\}} (-1)^{|J|}\int_{\prod_{k\in J} {[-M, M]}} {R_J(x)\frac{\partial \chi_J}{\partial x_J}dX_J}},
\]
where $|C_{\chi}| \leq 3vol(S^{r-1})(M\sqrt{r})^{r - 2(1 + \lfloor \frac{r}{2} \rfloor)} k ^{2(1 + \lfloor \frac{r}{2} \rfloor) - r}$ by Lemma \ref{coefficientbound}, and the sum is over the irreducible characters appearing in all tensor powers of fundamental representations of degree at most $2(1 + \lfloor \frac{r}{2} \rfloor)(k - 1) + r$. By Corollary \ref{corollary} this equals
\[
\sum_{\chi} {C_{\chi} \left[ \frac{1}{N} \sum_{i = 1}^N {\chi(g_i)} \right]},
\]
where the sum is over the same characters $\chi$, but with the trivial one removed,
since $\int {\chi d\mu} = 0$ for the nontrivial $\chi$, while for the trivial character $\chi_0$, $\int {\chi_0 d\mu} - \frac{1}{N} \sum_{i = 1}^N {\chi_0(g_i)} = 1 - 1 = 0$. In absolute value this is
\[
\leq 3vol(S^{r-1})(M\sqrt{r})^{r - 2(1 + \lfloor \frac{r}{2})} k ^{2(1 + \lfloor \frac{r}{2} \rfloor) - r} \sum_{\chi} \left| \frac{1}{N} \sum_{i = 1}^N \chi(g_i) \right|.
\]
Putting everything together we obtain
\begin{align}
|\int_{[-M, M]^r} {R(x)f_v(x)dX}| \leq 3vol(S^{r-1})(M\sqrt{r})^{r - 2(1 + \lfloor \frac{r}{2} \rfloor)}  k ^{2(1 + \lfloor \frac{r}{2} \rfloor) - r} \sum_{\chi} \left| \frac{1}{N} \sum_{i = 1}^N {\chi(g_i)}\right|.   \label{eq:34}
\end{align}
(\ref{eq:final11}) and (\ref{eq:34}) together imply that (Recall the case $D^* \leq Bk^{-1}$.)
\begin{align}
\label{eq:final2}
D^* \leq max\left \{ 300r5^r, \;    \frac{160r^2(5M)^r2^{r_2 + \frac{1}{2}(r + dim(G))}}{\pi^r } \right \} \left(\frac{1}{k} +  \sum_{\chi} |{\frac{1}{N} \sum_{i = 1}^N {\chi(g_i)}|} \right).
\end{align}
Since $M = \displaystyle \max_{\Gamma} dim(\Gamma)$, where the max is taken over fundamental representations $\Gamma$, Theorem \ref{maintheorem} now follows from the above inequality, inequality (\ref{eq:ineq3}), and the fact that $dim(G) \geq 3r$ (since $dim(G) \geq r + 2(\mbox{number of positive roots}) \geq  r + 2(\mbox{number of simple roots}) = 3r$), after we double the constant to get the $C_G$ in the theorem, to account for the fact that we've only considered odd $k$.  \hfill $\blacksquare$

\vspace{0.3in}

\noindent {\large \textbf{Acknowledgements}} I'd like to thank Nick Katz for pointing out Niederreiter's paper to me, and for many helpful conversations and suggestions.

\newpage

\begin{center}
{\LARGE References}
\end{center}

\noindent [B] N. Bourbaki. \emph{Lie Groups and Lie Algebras, Chapters 7-9}. Translated by Andrew Pressley. Springer, 2005.

\noindent [D-T]  Michael Drmota and Robert Tichy. \emph{Sequences, Discrepancies, and Applications}. Springer, 1997.

\noindent [N] Harald Niederreiter. \emph{The distribution of values of Kloosterman sums}. Arch. Math, Vol. 56, 270-277 (1991).

\end{document}